\documentclass[12pt]{amsart}

\usepackage{amsmath, amscd, graphicx, latexsym, hyperref, rlepsf, accents, times, enumerate}
\usepackage{hyperref}
\usepackage{tikz}
\usetikzlibrary{arrows,chains,matrix,positioning,scopes}

\input xy
\xyoption{all}

\newtheorem{theorem}{Theorem}[section]
\newtheorem{lemma}[theorem]{Lemma}
\newtheorem{proposition}[theorem]{Proposition}
\newtheorem{corollary}[theorem]{Corollary}

\theoremstyle{definition}
\newtheorem{definition}[theorem]{Definition}

\numberwithin{equation}{section}

\usepackage{color}

\numberwithin{equation}{section}

\makeatletter
\tikzset{join/.code=\tikzset{after node path={%
\ifx\tikzchainprevious\pgfutil@empty\else(\tikzchainprevious)%
edge[every join]#1(\tikzchaincurrent)\fi}}} \makeatother

\tikzset{>=stealth',every on chain/.append style={join},every
join/.style={->}}

\newcommand{\OB}{\mathcal{OB}}

\def\CP{\hbox{${\Bbb C} P^2$}}

\def\CPb{\hbox{$\overline{{\Bbb C}P^2}$}}

\def\v{\vskip.12in}

\def\a{\alpha}
\def\b{\beta}

\def\g{\gamma}

\begin{document}

\title[]{Singularity links with exotic Stein fillings}

\author[Akhmedov and Ozbagci]{Anar Akhmedov and Burak Ozbagci}

\address{School of Mathematics, University of Minnesota, Minneapolis, MN 55455, USA, \newline akhmedov@math.umn.edu}

\address{Department of Mathematics, Ko\c{c} University, Sariyer, 34450, Istanbul,
Turkey, \newline bozbagci@ku.edu.tr}



\begin{abstract}

In \cite{aems}, it was shown that there exist infinitely many
contact Seifert fibered $3$-manifolds each of which admits
infinitely many exotic (homeomorphic but pairwise non-diffeomorphic)
simply-connected Stein fillings. Here we extend this result to a
larger set of contact Seifert fibered $3$-manifolds with many
singular fibers and observe that these $3$-manifolds are singularity
links.  In addition, we prove that the contact  structures induced
by the Stein fillings are the canonical contact structures on these
singularity links.  As a consequence, we verify a prediction of
Andr\'{a}s N\'{e}methi \cite{nem} by providing examples of isolated
complex surface singularities whose links with their canonical
contact structures admitting infinitely many exotic simply-connected
Stein fillings. Moreover,  for infinitely many  of these contact
singularity links and for each positive integer $n$, we also
construct an infinite family of exotic Stein fillings with fixed
fundamental group $\mathbb{Z} \oplus \mathbb{Z}_n$.

\end{abstract}

\maketitle

\section{Introduction}
\label{sec:intro}

The link of a normal complex surface singularity  carries a Milnor
fillable (also known as canonical) contact structure $\xi_{can}$
which is uniquely determined up to isomorphism \cite{cnp}.  A Milnor
fillable contact structure is Stein fillable since a regular
neighborhood of the exceptional divisor in a resolution of the
surface singularity provides a holomorphic filling which can be
deformed to be a blow-up of a Stein surface without changing the
contact structure $\xi_{can}$ on the boundary \cite{bd}.  Moreover,
if a singularity admits a smoothing then the corresponding Milnor
fiber is also a Stein filling of $\xi_{can}$.

In this paper, we generalize the main result in \cite{aems} to a
larger family of contact Seifert fibered $3$-manifolds admitting
many singular fibers. We also observe an additional feature of these
contact $3$-manifolds: They appear as the links of some isolated
complex surface singularities, and the contact  structures are the
 canonical ones on these singularity links.  The following theorems verify a prediction of N\'{e}methi \cite{nem} filling a
gap in the literature.\\

\noindent{\bf Theorem 4.4.} \emph{There exist infinitely many
Seifert fibered singularity links each of which admits infinitely
many exotic simply-connected
Stein fillings of its canonical contact structure.} \\

\noindent{\bf Theorem 5.3.} \emph{There exists an infinite family of
Seifert fibered singularity links such that for each positive
integer $n$, each member of this family equipped with its canonical
contact structure admits infinitely many exotic Stein fillings whose
fundamental group is $\mathbb{Z}
\oplus \mathbb{Z}_{n}$.} \\

One should contrast our result with what is known for links of some
other isolated complex surface singularities. For example, Ohta and
Ono showed that the diffeomorphism type of any minimal strong symplectic filling of the link of
a simple singularity is unique which implies that the minimal resolution of the singularity is diffeomorphic to the Milnor fiber \cite{oo2}.  They also showed
that any minimal strong symplectic filling of the link of a simple elliptic singularity  is diffeomorphic
either to the minimal resolution or to the Milnor fiber of  the smoothing of the singularity  \cite{oo1} .

Moreover, Lisca showed that the canonical contact structure on a lens space (the
oriented link of some cyclic quotient singularity) has only
\emph{finitely} many distinct Stein fillings, up to diffeomorphism \cite{l} (see also earlier work of
McDuff \cite{mc}). Recently, it was
shown that these Stein fillings correspond bijectively to the Milnor fibres
coming from all possible distinct smoothings of the singularity \cite{npp}.

In summary, in all the previously studied examples in the literature, it was shown that an isolated complex surface singularity with its canonical contact structure admits finitely many diffeomorphism types of Stein fillings such that  each Stein filling is diffeomorphic
either to the minimal resolution or to the Milnor fiber of one of
the smoothings of the singularity.

We should mention that in \cite[Theorem 1.2]{oo} Ohta and Ono showed
the existence of singularity links which admit infinitely many
distinct minimal symplectic fillings distinguished by their
$b^{+}_2$. These fillings, however, are \emph{not necessarily Stein
or simply-connected}.

On the other hand, using log transforms, Akbulut and Yasui
\cite[Theorem 1.1]{ay} constructed contact $3$-manifolds admitting
infinitely many exotic simply-connected Stein fillings with $b_2=2$,
inspired by an earlier paper by Akbulut \cite{ak}. In these
articles, however,  the contact $3$-manifolds in question are
\emph{not singularity links}.

Finally, we would like to point out that in \cite{ako}, \emph{using
very different methods}, we were able to prove the statement of
Theorem 5.3 by replacing  $\mathbb{Z} \oplus \mathbb{Z}_{n}$ by any
finitely presented group $G$.

\section{Milnor fillable contact structures on Seifert fibered $3$-manifolds}
\label{sec:milseif}

In this section we identify the isomorphism class of the canonical
contact structure on a singularity link which admits a Seifert
fibration.  A topological characterization of such $3$-manifolds was
given by Neumann \cite{neu}: A closed and oriented Seifert fibered
$3$-manifold is a singularity link if and only if it has a Seifert
fibration over an orientable base such that the Euler number of this
fibration is negative.


\begin{proposition}  \label{mil} The isomorphism class of the Milnor fillable
contact structure $\xi_{can}$ on a closed and oriented $3$-manifold
$Y$ which has a Seifert fibration with negative Euler number over an
orientable base coincides with the unique isomorphism class of the
$S^1$-invariant transverse contact structures.
\end{proposition}

\begin{proof} It is known that any Milnor fillable contact structure $\xi_{can}$ on a singularity link is universally tight \cite{lo}.
According to \cite[Corollary 4]{Ma}, there exist a locally free
$S^1$-action on $Y$ such that $\xi_{can}$ is either transverse to
the orbits or invariant under the $S^1$-action. Moreover, a contact
structure which is both invariant and transverse to the orbits of a
locally free $S^1$-action exists on a Seifert fibered $3$-manifold
$Y$ exactly when the Euler number of $Y$ is negative (cf. \cite{kt,
lm}). Furthermore, there is only one isomorphism class of such
contact structures as indicated in the last paragraph on page 1356
in \cite{Ma} and hence the result follows since a Milnor fillable
contact structure is unique up to isomorphism \cite{cnp}.
\end{proof}

\section{Extending diffeomorphisms}\label{exte}

Let $\overline{p}=(p_1,p_2, \ldots, p_r)$ denote an $r$-tuple of
\emph{positive} integers and let $\Sigma_h$ denote a closed oriented
surface of genus $h \geq 0$. Let $Z_{h, \overline{p}}$ denote the
oriented smooth $4$-manifold-with-boundary obtained by plumbing
oriented disk bundles according to the star-shaped graph with $r+1$
vertices described as follows: The central vertex represents
$\Sigma_h \times D^2$ and if we label the remaining $r$ vertices by
$i=1, \ldots, r$, the $i$th vertex---connected by an edge to the
central vertex---represents a $D^2$-bundle over $S^2$ whose Euler
number is $-p_i$.

\begin{proposition} \label{ext} Any orientation preserving self-diffeomorphism of
$\partial Z_{h, \overline{p}}$ extends over $Z_{h, \overline{p}}$.
\end{proposition}

\begin{proof} We sketch the proof of this proposition which is a simple extension of the proof
of \cite[Lemma 3.1]{aems}, where the case $r = 1$ was treated in
full details. The strategy of the proof there, was to find the
required extension in two steps, where the first step was to find an
extension to the part $\accentset{\circ}{\Sigma_h} \times D^2$ of
the plumbing and then complete the extension on the remaining part.
Here $\accentset{\circ}{\Sigma_h}$ denotes $\Sigma_h$ with a disk
removed.

In order to prove our result, we apply the same strategy, where we
remove several disks from $\Sigma_h$ and the second paragraph in the
proof of \cite[Lemma 3.1]{aems} works verbatim as the initial step.
To complete the extension to the $r$ disk-bundles over $S^2$, we
rely on  a result of Bonahon \cite{bo}, since the boundary of the
oriented $D^2$-bundle over $S^2$ with Euler number $-p_i$ is the
oriented $S^1$-bundle over $S^2$ with the same Euler number, which
in turn, is orientation preserving diffeomorphic to the lens space
$L(p_{i},1)$.
\end{proof}

\section{Singularity links with simply-connected exotic Stein
fillings}\label{iso}

The boundary $\partial Z_{h, \overline{p}}$ has an orientation
induced from  the orientation on the smooth
$4$-manifold-with-boundary $Z_{h, \overline{p}}$ described in
Section~\ref{exte}. Let $Y_{h,\overline{p}}$ denote $\partial Z_{h,
\overline{p}}$  with the \emph{opposite} orientation. In other
words, $Y_{h,\overline{p}}$ is the closed, oriented $3$-manifold
which is obtained by plumbing of oriented circle bundles according
to the star-shaped graph with $r+1$ vertices as illustrated on the
left in Figure~\ref{star}: The central vertex represents $\Sigma_h
\times S^1$  and if we label the remaining $r$ vertices by $i=1,
\ldots, r$, the $i$th vertex---connected by an edge to the central
vertex---represents an $S^1$-bundle over $S^2$ whose Euler number is
$p_i$.

\begin{lemma}\label{sin} The $3$-manifold $Y_{h,\overline{p}}$ is the link of an
isolated complex surface singularity.

\end{lemma}
\begin{proof}

The $3$-manifold $Y_{h,\overline{p}}$ is obtained by plumbing of
circle bundles according to the star-shaped graph illustrated on the
left in Figure~\ref{star} with $r+1$ vertices, where the weight on a
vertex represents the Euler number of the corresponding oriented
circle bundle as usual.

\begin{figure}[ht]
  \relabelbox \small {
  \centerline{\epsfbox{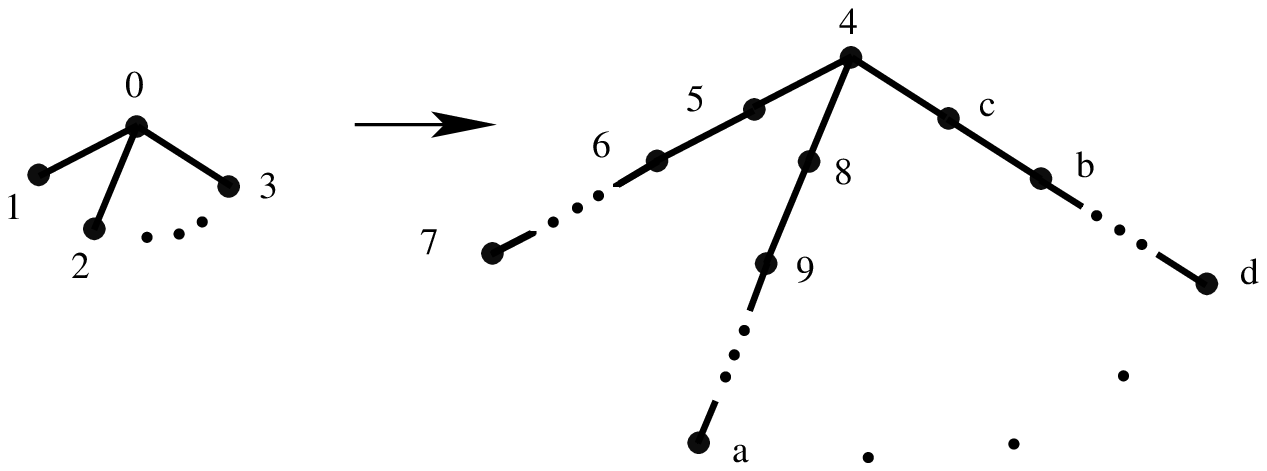}}}

\relabel{0}{$0$}

\relabel{1}{$p_1$}

\relabel{2}{$p_2$}

\relabel{3}{$p_r$}

\relabel{4}{$-r$}

\relabel{5}{$-2$}

\relabel{6}{$-2$}

\relabel{7}{$-2$}

\relabel{8}{$-2$}

\relabel{9}{$-2$}

\relabel{a}{$-2$}

\relabel{b}{$-2$}

\relabel{c}{$-2$}

\relabel{d}{$-2$}

\endrelabelbox
        \caption{All except the central vertex
represent circle bundles over $S^2$.}
        \label{star}
\end{figure} By blowing up and down this plumbing graph several times we see that
$Y_{h,\overline{p}}$ is orientation-preserving diffeomorphic to the
$3$-manifold given by the star-shaped plumbing graph depicted on the
right in Figure~\ref{star}, where there are $r$ legs emanating from
the central vertex of weight $-r$ (and genus $h$) and the $i$-th leg
is a chain of $p_i -1$ vertices (excluding the central vertex) each
with weight $-2$.
 Since the intersection matrix of
this last graph is negative definite, we conclude that
$Y_{h,\overline{p}}$ is orientation-preserving diffeomorphic to the
link of a normal and hence isolated surface singularity by Grauert's
theorem. \end{proof}

Let $\OB_{h, \overline{p}}$ denote the open book on $Y_{h,
\overline{p}}$ whose page is a genus $h \geq 0$ surface with $r \geq
1$ boundary components and monodromy is given as $$t_1^{p_1}
t_2^{p_2} \ldots t_r^{p_r}$$ where $t_i$ is a right-handed Dehn
twist along a curve parallel to the $i$-th  boundary component and
let $\xi_{h,\overline{p}}$ denote  the contact structure which is
supported by $\OB_{h, \overline{p}}$.

\begin{lemma} \label{can} The contact structure $\xi_{h,\overline{p}}$ is
the canonical contact structure on $Y_{h,\overline{p}}$.

\end{lemma}
\begin{proof}

First we observe that $Y_{h,\overline{p}}$ admits a Seifert
fibration over a closed oriented surface of genus $h$ with $r$
singular fibers with multiplicities $p_1,p_2,\ldots,p_r$,
respectively. Note that an explicit open book transverse to the
fibers of such a Seifert fibration was constructed in \cite{o},
which is indeed isomorphic to the open book $\OB_{h,\overline{p}}$
on $Y_{h,\overline{p}}$. Moreover, it was also shown \cite{o} that
the contact structure supported by this open book is transverse to
the Seifert fibration. Furthermore, it is easy to see that this
contact structure is invariant under the natural $S^1$ action
induced by the fibration. This is because the pages of the open book
are $S^1$-invariant by construction and contact planes can be
perturbed to be arbitrarily close to tangents of the pages by
allowing an isotopy of the contact structure \cite{et}.  Therefore
$\xi_{h,\overline{p}}$ has to be the unique Milnor fillable contact
structure on $Y_{h,\overline{p}}$ by
Proposition~\ref{mil}.\end{proof}

The following was proved in \cite{ao2}:

\begin{proposition} \label{palf} Suppose that the closed $4$-manifold $X$ admits
a genus $h$ Lefschetz fibration over $S^2$ with homologically
nontrivial vanishing cycles. Let $S_1, S_2,\ldots, S_r$ be $r$
disjoint sections of this fibration with squares $-p_1,-p_2,\ldots,
-p_r$, respectively.  Let $V$ denote the $4$-manifold with boundary
obtained from $X$ by removing a regular neighborhood of these $r$
sections union a nonsingular fiber. Then $V$ admits a PALF (positive
allowable Lefschetz fibration over $D^2$)  and hence a Stein
structure such that the induced contact structure
$\xi_{h,\overline{p}}$ on $\partial V=Y_{h,\overline{p}}$ is
compatible with the open book  $\OB_{h,\overline{p}}$ induced by
this PALF, where $\overline{p}=(p_1,p_2, \ldots, p_r)$. In other
words, $V$ is a Stein filling of the contact $3$-manifold
$(Y_{h,\overline{p}}\;,\; \xi_{h,\overline{p}}) $.
\end{proposition}

Now we are ready to state and prove the main result of this section: \\

\begin{theorem} \label{sim} There exist infinitely many Seifert fibered
singularity links each of which admits infinitely many exotic
(homeomorphic but pairwise non-diffeomorphic) simply connected Stein
fillings of its canonical contact structure.
\end{theorem}

\begin{proof} We will give a proof of this result in the following four
parts: \\

\textbf{\underline{Part 1}.} \emph{A genus $g$ Lefschetz fibration
on $\CP\#(4g+5)\CPb$}: Let $\Sigma_g$ be a closed orientable surface
of genus $g \geq 1$. Let $\gamma_1$, $\gamma_2$, \ldots,
$\gamma_{2g+1}$ denote the collection of simple closed curves on
$\Sigma_g$ depicted in Figure~\ref{hyp}, and $c_{i}$ denote the
right handed Dehn twists along the curve $\gamma_i$. Let $X(g,1)$
 denote $\CP\#(4g+5)\CPb$. The next result is well-known (cf. \cite[Exercises 7.3.8(b) and
8.4.2(a)]{gos}).

\begin{lemma} There is a hyperelliptic genus $g$ Lefschetz fibration $f_1:  X(g,1) \to
S^2$ with global monodromy $ (c_1c_2 \cdots
c_{2g-1}c_{2g}{c^2_{2g+1}}c_{2g}c_{2g-1} \cdots c_2c_1)^2 = 1.$
\end{lemma}

\begin{figure}[ht]
  \relabelbox \small {
  \centerline{\epsfbox{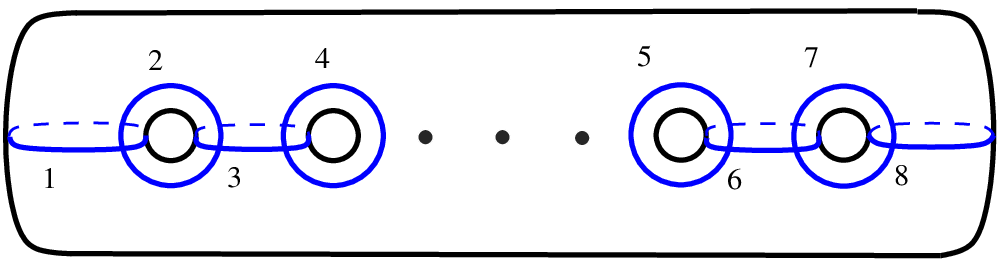}}}

\relabel{1}{$\g_1$}

\relabel{2}{$\g_2$}

\relabel{3}{$\g_3$}

\relabel{4}{$\g_4$}

\relabel{5}{$\g_{2g-2}$}

\relabel{6}{$\g_{2g-1}$}

\relabel{7}{$\g_{2g}$}

\relabel{8}{$\g_{2g+1}$}

  \endrelabelbox
        \caption{Vanishing cycles of the hyperelliptic  genus $g$ Lefschetz fibration
        $f_1: X(g,1)= \CP\#(4g+5)\CPb \to S^2$}
        \label{hyp}
\end{figure}

The $4$-manifold  $X(g,1)$ is diffeomorphic to the desingularization
of the double branched cover of $S^2 \times S^2$ with branch locus
given as two copies of $S^2 \times pt$ and $2g+2$ copies of $pt
\times S^2$. Based on this description, it is easy to see that
$X(g,1)$ admits a ``vertical" genus $g$ fibration with two singular
fibers and a ``horizontal" fibration with $S^2$ as a regular fiber.
Moreover the vertical fibration can be locally perturbed so that it
becomes isomorphic to the Lefschetz fibration $f_1:  X(g,1) \to S^2$
as explained in \cite[Exercise 8.4.2(c)]{gos}).

\begin{lemma} \cite[Corollary 4.6]{t}\label{tana} For any $g \geq 1$, $f_1 : X(g,1) \to S^2$
admits at least $4g + 4$ disjoint $(-1)$-sphere sections.
\end{lemma}

We claim that the exceptional sphere $s_i$ of the $i$-th blow up is
a section of the Lefschetz fibration $f_1 : X(g,1) \to S^2$ for $2
\leq i \leq 4g+5$. Let $h$ denote the canonical generator of $H_2
(\CP, \mathbb{Z})$ and let $[F] \in H_2 (X(g,1), \mathbb{Z})$ denote
the homology class of the fiber $F$ of the Lefschetz fibration $f_1
: X(g,1) \to S^2$. Then, by \cite[Lemma 3.3]{fps}, we have
$[F]=(g+2)h - ge_1 - e_2 -\cdots - e_{4g+5}$, where $e_i=[s_i]$
denotes the homology class of the sphere $s_i$. Since, $[F] \cdot
e_i=1$ (for $2 \leq i \leq 4g+5$) and the fiber $F$ and sphere $e_i$
can be chosen to be pseudo-holomorphic (so that they only intersect
positively), we conclude that the exceptional sphere $s_i$
intersects each genus $g$ fiber of the Lefschetz fibration $f_1 :
X(g,1) \to S^2$ geometrically ones---which proves our claim.

Note that the fiber of the horizontal fibration above is a square
zero sphere in $X(g,1)$ given by the homology class $h - e_{1}$,
which intersects every fiber of $f_1: X(g,1) \to S^2$ twice.

\begin{definition} We denote the $n$-fold fiber sum of the genus $g$ Lefschetz
fibration $f_1: X(g,1) \to S^2$ by $f_n: X(g,n) \to S^2$.
\end{definition}

By sewing together the disjoint $(-1)$-sphere sections of $f_1:
X(g,1) \to S^2$ we obtain $4g+4$ disjoint $(-n)$-sphere sections of
$f_n: X(g,n) \to S^2$. In order to prove Theorem~\ref{sim}, we just
focus on $f_2: X(g,2) \to S^2$ for $g \geq 2$. When we fiber sum two
copies of $f_1: X(g,1) \to S^2$ to obtain $f_2: X(g,2) \to S^2$, we
can also glue together
 square-zero sphere fibers of the horizontal fibrations on each
summand to construct an embedded essential torus $T$ of square zero
in $X(g,2)$. The outcome of Part 1 of our proof is that

\begin{lemma}\label{t} There is  an embedded torus $T$ in $X(g,2)$ with two key properties:
(i) $T$ intersects every fiber of the genus $g$ Lefschetz fibration
$f_2: X(g,2) \to S^2$ at two points and (ii)  $T$ has no
intersection with the $4g+4$ disjoint $(-2)$-sphere sections of this
fibration.
\end{lemma}

\textbf{\underline{Part 2}.} \emph{Fintushel-Stern knot surgery}:
Let $X(g,2)_{K}$ denote the $4$-manifold obtained from $X(g,2)$ by
performing a Fintushel-Stern knot surgery on the torus $T$ (see
Lemma~\ref{t}) in $X(g,2)$ using a  knot $K \subset S^3$ (cf.
\cite{fs}). More precisely, $$X(g,2)_{K} = (X(g,2) \setminus (T
\times D^2)) \cup (S^1 \times (S^3 \setminus N(K)),$$ where  we
identify the boundary of a disk normal to $T$ with a longitude of a
tubular neighborhood $N(K)$  of  $K$ in $S^3$. Next we observe that,

\begin{lemma} \label{lft} For any genus $k$ fibered knot $K$, the surgered $4$-manifold
$X(g,2)_{K}$ admits a genus $(g+2k)$-Lefschetz fibration with $4g+4$
disjoint $(-2)$-sphere sections.
\end{lemma}

\begin{proof} The torus $T \subset X(g,2)$ on
which we perform knot surgery intersects  every fiber of  $f_2:
X(g,2) \to S^2$ twice and  a  fiber of the Lefschetz fibration
$X(g,2)_K \to S^2$ is obtained by gluing one copy of the Seifert
surface of the fibered knot $K$ to each puncture of the twice
punctured fiber of $f_2: X(g,2) \to S^2$ (cf. \cite{fs2}).

Recall that $e_2, e_3, \ldots, e_{4g+5}$ denote the homology classes
of the disjoint $(-1)$-sphere sections of $f_1: X(g,1) \to S^2$.
When we fiber sum two copies of $f_1: X(g,1) \to S^2$, we can glue
corresponding $(-1)$-sphere sections in the two summands to obtain
$4g+4$ disjoint $(-2)$-sphere sections $S_2, S_3, \ldots, S_{4g+5}$
of $f_2: X(g,2) \to S^2$. Note that these $(-2)$-sphere sections
will remain as sections of the Lefschetz fibration $X(g,2)_{K} \to
S^2$, since they are disjoint from the surgery torus $T$.
\end{proof}

\textbf{\underline{Part 3}.} \emph{Construction of the
simply-connected Stein fillings}:

\begin{definition}  For any $ 1\leq r \leq 4g+3$ and for any genus $k$ fibered knot $K$ in $S^3$, the $4$-manifold-with-boundary  $V(g,r)_{K} \subset X(g,2)_{K} $ is obtained by removing  a regular
neighborhood of $r$ disjoint sections $S_{2}, S_{3}, \ldots,
S_{r+1}$ union a nonsingular genus $g + 2k$ fiber of the Lefschetz
fibration  $X(g,2)_{K} \to S^2$ given in Lemma~\ref{lft}.
\end{definition}

We would like to emphasize that we do not remove the section
$S_{4g+5}$.

\begin{lemma}
The $4$-manifold  $V(g,r)_{K}$ is simply-connected.
\end{lemma}

\begin{proof}

Observe that, by the Seifert-Van Kampen's theorem, the fundamental
group of $V(g,r)_{K}$ is generated by the homotopy classes of loops
based at some point $q \in S_{4g+5}$ that are conjugate to loops
$\mu_{2}, \mu_{3}, \ldots, \mu_{r+1}$ and $\eta$ normal to $S_2,
S_3, \ldots, S_{r+1}$, and to the regular fiber we remove,
respectively. Since all the loops $\mu_{2}, \mu_{3}, \ldots,
\mu_{r+1}$, and $\eta$ can be deformed to a point using the spheres
represented by the homology classes $e_{4g+5} - e_{2}, e_{4g+5} -
e_{3}, \ldots, e_{4g+5} - e_{r+1}$ and the section $S_{4g+5}$,
respectively, we conclude that $V(g,r)_{K}$ is simply-connected.
\end{proof}

For any positive integer $r$, let $\overline r$ denote the $r$-tuple
$(2,2,\ldots,2)$ for the rest of this section. Then
Proposition~\ref{palf} coupled with Lemma~\ref{can} imply that

\begin{lemma}
The $4$-manifold $V(g,r)_{K}$ is a Stein filling of
$(Y_{g+2k,\overline{r}}\;,\;\xi_{g+2k,\overline{r}})$, where
$\xi_{g+2k,\overline{r}}$ is the canonical contact structure on the
Seifert-fibered singularity link $Y_{g+2k,\overline{r}}$.
\end{lemma}

\textbf{\underline{Part 4}.} \emph{An infinite family of exotic
Stein fillings}: For $k \geq 2$, let $\mathcal{F}_k=\{K_{k,i} : i
\in \mathbb{N} \}$ denote  an infinite family of genus $k$ fibered
knots in $S^3$ with pairwise distinct Alexander polynomials, which
exists by \cite{kan}. Then the infinite family $\{ X(g,2)_{K_{k,i}}
: K_{k,i} \in \mathcal{F}_k\}$ consists of smooth $4$-manifolds
homeomorphic to $X(g,2)$ which are pairwise non-diffeomorphic by a
theorem of Fintushel and Stern \cite{fs}. Now we claim that for
fixed $g \geq 2$, $k \geq 2$, and $1 \leq r \leq 4g+3$, the infinite
set
$$\mathcal{S}_{g,k,r} = \{ V(g,r)_{K_{k,i}} :  K_{k,i} \in
\mathcal{F}_k\}$$ indexed by $i \in \mathbb{N}$, includes infinitely
many homeomorphic but pairwise non-diffeomorphic simply-connected
Stein fillings of the Seifert fibered singularity link
$(Y_{g+2k,\overline{r}}\;,\;\xi_{g+2k,\overline{r}})$.

In order to prove that these Stein fillings are pairwise
non-diffeomorphic we just appeal to Proposition~\ref{ext}, by
observing that what we delete from $X(g,2)_{K_{k,i}}$ to obtain
$V(g,r)_{K_{k,i}}$ is indeed diffeomorphic to $Z_{g+2k,
\overline{r}}$.

Next we prove that infinitely many of  the Stein fillings in
$\mathcal{S}_{g,k,r}$ are homeomorphic.  We first observe that all
of these Stein fillings have the same Euler characteristic by
elementary facts and the same signature by Novikov additivity. It
follows that the rank of the second homology group of the fillings
is fixed as well because the fillings are simply-connected.
Moreover, since the boundary of any Stein filling in
$\mathcal{S}_{g,k,r}$ is diffeomorphic to $Y_{g+2k,\overline{r}}$
and $H_1(Y_{g+2k,\overline{r}}\; ; \mathbb{Z})$ is infinite, we
conclude that the determinant of the intersection form of any
filling in $\mathcal{S}_{g,k,r}$ is zero. It follows that
intersection forms of all the Stein fillings  in
$\mathcal{S}_{g,k,r}$ are isomorphic (see \cite[Corollary 5.3.12 and
Exercise 5.3.13(f)]{gos}). Furthermore, a fixed symmetric bilinear
form is realized as an intersection form by only finitely many
homeomorphism types of simply-connected compact oriented
$4$-manifolds with a given boundary \cite[Corollary 0.4]{b}.
Therefore the infinitely many Stein fillings in
$\mathcal{S}_{g,k,r}$ belong to finitely many homeomorphism
types---which finishes the proof of Theorem~\ref{sim}.\end{proof}

\section{Exotic Stein fillings with non-trivial fundamental groups}\label{abelianstein}

Our aim in this section is to explore the existence of non-simply
connected exotic Stein fillings of some singularity links. Let $n$
be a positive integer.  In this paper, we only study the case when
the fundamental group of the Stein fillings is $\mathbb{Z}\oplus
\mathbb{Z}_n$.

As an essential ingredient we use the family of non-holomorphic
genus $g$ Lefschetz fibrations with fundamental group $\mathbb{Z}
\oplus \mathbb{Z}_{n}$ constructed in \cite{os1} for $g =2$ and
generalized to the case $g \geq 3$ in \cite{k}. For the purposes of
this article we focus on the case where $g \geq 3$ is odd and
provide the necessary background for the convenience of the reader.

\begin{definition}
Let $W(m):= \Sigma_{m} \times \mathbb{S}^2\#8\CPb$, where $\Sigma_m$
denotes a closed oriented genus $m$ surface.

\end{definition}

Note that $W(m)$ is the total space of a genus $g = 2m+1$ Lefschetz
fibration over $\mathbb{S}^2$, which was proved in \cite[Remark
5.2]{k} generalizing a classical result for $g=2$ due to Y.
Matsumoto. The branched-cover description of this Lefschetz
fibration can be given as follows \cite{fs2}: Take a double branched
cover of $\Sigma_{m} \times \mathbb{S}^2$ along the union of four
disjoint copies of ${pt}\times \mathbb{S}^{2}$ and two disjoint
copies of $\Sigma_{m} \times {pt}$ as shown in Figure~\ref{bloc}.

\begin{figure}[ht]
  \relabelbox \small {
  \centerline{\epsfbox{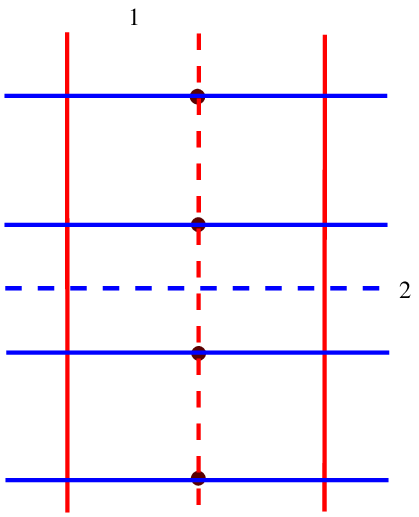}}}

\relabel{1}{$\Sigma_m \times pt$}

\relabel{2}{$pt \times S^2$}

  \endrelabelbox
        \caption{The branch set in $\Sigma_{m} \times
\mathbb{S}^2$}
        \label{bloc}
\end{figure}

The resulting branched cover has eight singular points,
corresponding to the intersection points of the horizontal spheres
and the vertical genus $m$ surfaces in the branch set. By
desingularizing this singular manifold one obtains $W(m)$. Observe
that a generic fiber of the vertical fibration is the double cover
of $\Sigma_{m}$ branched over four points. Thus, a generic fiber is
a genus $g=2m+1$ surface and each of the two singular fibers of the
vertical fibration can be perturbed into $2m+6$ Lefschetz type
singular fibers.

\begin{proposition}\cite{k} The $4$-manifold $W(m)$ admits a genus $g$ Lefschetz fibration over
$S^2$ with $2g+10$ singular fibers such that the monodromy of this
fibration is given by the relation
$$(b_0 b_1 b_2 \ldots b_g a^2b^2)^2=1$$ where
$b_i$ denotes a right-handed Dehn twists along $\b_i$, for
$i=0,1,\ldots,g$ and $a$ and $b$ denote right-handed Dehn twists
along $\a$ and $\b$ respectively (see Figure~\ref{rel}).
\end{proposition}

\begin{figure}[ht]
  \relabelbox \small {
  \centerline{\epsfbox{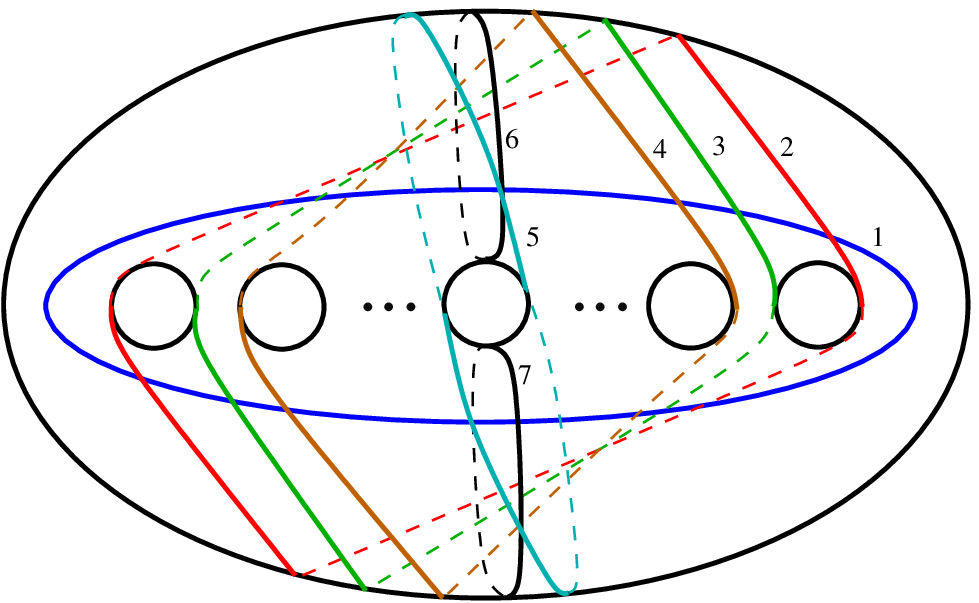}}}

\relabel{1}{$\beta_0$}

\relabel{2}{$\beta_1$}

\relabel{3}{$\beta_2$}

\relabel{4}{$\beta_3$}

\relabel{5}{$\beta_g$}

\relabel{6}{$\alpha$}

\relabel{7}{$\beta$}

\endrelabelbox
        \caption{Vanishing cycles of the genus $g=2m+1$ Lefschetz fibration
        $W(m) =\Sigma_{m}\times \mathbb{S}^2\#8\CPb  \to S^2$ corresponding to
$(b_0 b_1 b_2 \ldots b_g a^2b^2)^2=1.$}
        \label{rel}
\end{figure}

Also, a generic fiber of the horizontal fibration is the double
cover of $\mathbb{S}^2$ branched over two points. This gives a
sphere fibration on $W(m)$. We are now ready to state the main
result of this section.

\begin{theorem} \label{sima} There exists an infinite family of
Seifert fibered singularity links such that for each positive
integer $n$, each member of this family equipped with its canonical
contact structure admits infinitely many exotic (homeomorphic but
pairwise non-diffeomorphic) Stein fillings whose fundamental group
is $\mathbb{Z} \oplus \mathbb{Z}_{n}$.
\end{theorem}

\begin{proof}

For $g = 2m+1 \geq 3$, let $W_n(m)$ denote the total space of the
Lefschetz fibration over $S^2$ obtained by a \emph{twisted} fiber
sum of two copies of the Lefschetz fibration $W(m) \to S^2$ along
the regular genus $g$ fiber (cf.  \cite{os1, k}).  Notice that
twisted fiber sum refers to the fiber sum where a regular
neighborhood of a fixed regular fiber of $W(m) \to S^2$ is
identified with a regular neighborhood of a fixed regular fiber of
another copy of $W(m) \to S^2$ by a non-trivial diffeomorphism of
the fiber. As shown in \cite{ako}, there is a diffeomorphism
involving an $n$-fold power of a right-handed Dehn twist along a
homologically nontrivial curve on the fiber such that
$\pi_1(W_n(m))= \mathbb{Z} \oplus \mathbb{Z}_n$. Since in $W(m)$ the
generic fiber of the vertical fibration intersects the generic fiber
of the sphere fibration in two points, after the fiber sum we have
an embedded homologically essential torus $T$ of self-intersection
zero in $W_n(m)$. Notice that a regular fiber of the genus $g$
fibration on $W_n(m)$ intersects $T$ at two points. It was shown in
\cite{k2} that the Lefschetz fibration on $W(m)$ admits at least two
disjoint $(-1)$-sphere sections, which implies that the Lefschetz
fibration on $W_n(m)$ admits at least two disjoint $(-2)$-sphere
sections. The torus $T$ above can be chosen to be disjoint from
these $(-2)$-sphere sections.

Let $W_n(m)_K$ denote the result of the Fintushel-Stern knot surgery
along the torus $T$ by a knot $K$ in $S^3$.  We observe that  by
Seifert-Van Kampen's theorem, $\pi_1 (W_n(m)_{K})= \mathbb{Z}\oplus
\mathbb{Z}_{n}$, since all the loops on $T$ are nullhomotopic in
$W_n(m)$ and  $W_n(m)_K$.

\begin{proposition} \label{hom} For any pair of positive integers $(m, n)$ and for any knot $K$ in $S^3$, the $4$-manifold
$W_n(m)_K$ is homeomorphic to $W_n(m)$.
 \end{proposition}

\begin{proof} The branched-cover description of the $4$-manifold
$W(m)$, whose branch locus in $\Sigma_{m} \times \mathbb{S}^2$ is
depicted in Figure~\ref{bloc}, shows that $W(m)$ admits a sphere
fibration, and the generic fiber of the genus $g$ Lefschetz
fibration on $W(m)$ intersects the generic fiber of the sphere
fibration in two points. Hence the untwisted fiber sum of two copies
of the Lefschetz fibration $W(m) \to S^2$ along the regular genus
$g$ fiber, which we denote by $W_0(m)$, admits an elliptic
fibration. Alternatively, $W_0(m)$ can be viewed as the fiber sum of
$\Sigma_{m} \times T^{2}$ and the elliptic surface $E(2)$ where we
identify $pt \times T^{2} \subset \Sigma_{m} \times T^{2}$ with an
elliptic fiber of $E(2)$. The elliptic fibration structure on
$W_0(m)$ over the genus $g$ surface is induced from the elliptic
fibrations of $E(2)$ and $\Sigma_{m} \times T^{2}$ via this fiber
sum. Moreover, the manifold $W_n(m)$ can be obtained from $W_0(m)$
by a single Luttinger surgery along a Lagrangian torus (for details,
see \cite{ako}), disjoint from an elliptic fiber. Therefore,
$W_n(m)$ contains a Gompf nucleus $C_{2}$ of $E(2)$: Use a cusp
fiber of the above mentioned elliptic fibration, and a $(-2)$-sphere
section obtained by sewing together $(-1)$-sphere sections of the
sphere fibration on $W(m)$. Furthermore, the torus along which we
perform Fintushel-Stern knot surgery can be assumed to lie in this
cusp neighborhood.

Next we  decompose $W_n(m)$ into $C_{2} \cup_{\Sigma(2,3,11)}
V(n,m)$ along the homology $3$-sphere $\Sigma(2,3,11)$,  where
$V(n,m)$ denotes the complement of $C_{2}$. Then, for any knot $K$
in $S^3$, we have a corresponding decomposition of $W_n(m)_{K}$ into
$(C_{2})_{K} \cup_{\Sigma(2,3,11)} V(n,m)$, where $(C_{2})_{K}$ is
an exotic copy of $C_{2}$ (cf. \cite{go}). Since $\partial C_2$ is a
homology $3$-sphere, by \cite[Corollary 0.9]{b}, there exits a
homeomorphism from $(C_{2})_{K}$ to  $C_{2}$ which restricts to the
identity map on the boundary. As a consequence, we have constructed
a homeomorphism between the $4$-manifolds $W_n(m)_{K}$ and $W_n(m)$
which extends the identity map on $V(n,m)$. \end{proof}

Suppose that $K$ is a fibered knot in $S^3$  of genus $k$. Then, a
simple argument similar to the one used in the proof of
Lemma~\ref{lft} shows that $W_n(m)_{K}$ admits a genus $g + 2k = 2(m
+ k) +1$ Lefschetz fibration over $S^2$ with two disjoint
$(-2)$-sphere sections. Recall that, in Part $4$ of the proof of
Theorem~\ref{sim}, for any $k \geq 2$, we denoted an infinite family
of genus $k$ fibered knots  in $S^3$  with pairwise distinct
Alexander polynomials by $\mathcal{F}_k=\{K_{k,i} : i \in \mathbb{N}
\}$.

Now let us fix a triple of positive integers $(m,n,k)$, where $k
\geq 2$. By removing a tubular neighborhood of a regular fiber and
only one of the two $(-2)$-sphere sections of the genus $2(m+k)+1$
Lefschetz fibration on $W_n(m)_{K_{k,i}}$, we obtain an infinite
family (indexed by $i \in \mathbb{N}$) of Stein fillings of the
Seifert fibered singularity link $Y_{2(m+k)+1, (2)}$ with its
canonical contact structure, such that each filling has $\pi_1 =
\mathbb{Z} \oplus \mathbb{Z}_{n}$. We claim that these Stein
fillings are exotic copies of each other, i.e., they are all
homeomorphic but pairwise non-diffeomorphic. The fact that these
fillings are pairwise non-diffeomorphic follows from
Proposition~\ref{ext} as in Part $4$ in the proof of
Theorem~\ref{sim}.

Finally, for fixed $(m,n,k)$, we show that the Stein fillings
described above with $\pi_1 = \mathbb{Z} \oplus \mathbb{Z}_{n}$
belong to the same homeomorphism type. We proved in
Proposition~\ref{hom} that for fixed positive integers $m$ and $n$,
all the smooth $4$-manifolds in the infinite family
$\{W_n(m)_{K_{k,i}} : K_{k,i} \in \mathcal{F}_k\}$ belong to the
same homeomorphism type, independent of the knot $K_{k,i}$. Now we
simply claim that the knot surgery performed on $W_n(m)$ to obtain
$W_n(m)_{K_{k,i}}$ essentially affects the complement of the removed
neighborhood of the regular fiber union the $(-2)$-sphere section,
and hence it does not have any effect on the homeomorphism type of
the ``remaining" Stein fillings. So the strategy is to start with a
homeomorphism of the closed $4$-manifolds including the Stein
fillings, and verify that it will ``descend" to a homoeomorphism of
the Stein fillings when we remove a piece from each after performing
a Fintushel-Stern knot surgery.

More precisely, first note that in $W_n(m)$, a tubular neighborhood
of the $(-2)$-sphere section is disjoint from the cusp neighborhood
(see the proof of Proposition~\ref{hom}) including the torus $T$
given above. Moreover, the cusp neighborhood intersects with a
tubular neighborhood of a regular fiber along two disjoint copies of
$D^2 \times D^2$. Since the initial homeomorphism in
Proposition~\ref{hom} is identity on the complement of the cusp
neighborhood, we can delete these configurations, except the two
copies of $D^2 \times D^2$, without affecting our homeomorphism.
Performing knot surgery on $T$ turns these two disk bundles into
$\accentset{\circ}{\Sigma}_k \times D^2$, where
$\accentset{\circ}{\Sigma}_k$ denotes a genus $k$ surface with one
disk removed. Since any homeomorphism of
$\partial(\accentset{\circ}{\Sigma_k} \times D^2)$ extends, we can
delete these two $D^2 \times \accentset{\circ}{\Sigma}_k$ as well so
that the homeomorphism descends to the Stein fillings. \end{proof}

\begin{corollary} For each $h \geq 7$, the Seifert fibered singularity link
$Y_{h,(2)}$ with its canonical contact structure $\xi_{h,(2)}$
admits
\begin{itemize}
\item an infinite family of exotic simply-connected Stein fillings,

\item for each positive integer $n$, an infinite family of exotic
Stein fillings whose fundamental group is $\mathbb{Z} \oplus
\mathbb{Z}_{n}$, and

\item for each positive integer $n$, a Stein filling whose first
homology group is $\mathbb{Z}^{h-2} \oplus \mathbb{Z}_{n}$.
\end{itemize} In particular, none of the Stein fillings in the last two items are
homeomorphic to a Milnor fiber of the singularity.
\end{corollary}

\begin{proof} Recall that, with respect to our notation in Section~\ref{iso},
$Y_{h,(2)}$ denotes the plumbing of $\Sigma_h \times D^2$  with an
oriented circle bundle over $S^2$ whose Euler number is $2$. It
follows that $Y_{h,(2)}$ is a Seifert fibered $3$-manifold over a
genus $h$ surface with a unique singular fiber of multiplicity $2$.

For any $h \geq 6$, an infinite family of simply connected,
homeomorphic but pairwise non-diffeomorphic Stein fillings of the
singularity link $(Y_{h,(2)}, \xi_{h,(2)})$ is given in
Theorem~\ref{sim}. Similarly, according to Theorem~\ref{sima}, for
any $h=g+2k\geq 7 $, and for each positive integer $n$, $(Y_{h,(2)},
\xi_{h,(2)})$ admits an infinite family of homeomorphic but pairwise
non-diffeomorphic Stein fillings with fundamental group $\mathbb{Z}
\oplus \mathbb{Z}_{n}$. The third family of Stein fillings is given
in \cite{os1}. In addition, none of the Stein fillings in the last
two items are homeomorphic to any Milnor fiber of the singularity,
since a Milnor fiber of a normal surface singularity has vanishing
first Betti number \cite{grs}.\end{proof}

\v  \noindent {\bf {Acknowledgement}}:  The authors would like to
thank the referee for his careful reading of the manuscript and
his/her suggestions that improved the presentation greatly.  This
work was initiated at the FRG workshop ``Topology and Invariants of
Smooth 4-manifolds" in Miami, USA and was completed  at the
``Invariants in Low-Dimensional Topology and Knot Theory" workshop
held in the Oberwolfach Mathematics Institute in Germany. We are
very grateful to the organizers of both workshops for creating a
very stimulating environment.  A. A. was partially supported by NSF
grants FRG-0244663, DMS-1005741 and a Sloan Fellowship. B.O. was
partially supported by the Marie Curie International Outgoing
Fellowship 236639.

\end{document}